\documentclass[12pt]{amsart}
\usepackage{amsthm,amssymb,amsmath,enumerate,graphicx}
\usepackage{amscd}
\usepackage{setspace}
\theoremstyle{plain}
\newcommand{\zd}{{\mathcal D}}

\newcommand{\zf}{{\mathcal F}}
\newcommand{\zg}{{\mathcal G}}

\newcommand{\zi}{{\mathcal I}}

\newcommand{\zv}{{\mathcal V}}
\newcommand{\tg}{\tilde{\gamma}}
\newcommand{\nsw}{{({\nu}^w)}^*}
\newcommand{\R}{\mathbb{R}}
\newtheorem {theorem}{Theorem}[section]
\newtheorem {lemma}[theorem]{Lemma}
\theoremstyle{remark}
\newtheorem {definition}[theorem]{Definition}

\newtheorem {example}[theorem]{Example}

\newtheorem {Conjecture}[theorem]{Conjecture}

\newtheorem {notation}[theorem]{Notation}
\newtheorem {remark}[theorem]{Remark}
 \newtheorem{assertion}{Assertion}
\newtheorem*{assertion*}{Assertion}
\begin{document}
\begin{abstract}
Given a partition $\zv=(V_1, \ldots,V_m)$ of the vertex set of a graph $G$, an {\em independent transversal} (IT) is an independent set in $G$ that contains one vertex from each $V_i$.
A {\em fractional IT} is a non-negative real valued function on $V(G)$ that represents each part with total weight at least $1$, and belongs as a vector to the convex hull of the incidence vectors of independent sets in the graph.
It is known that if the  domination number of the graph induced on the union of every $k$ parts $V_i$ is at least $k$, then there is a fractional IT. We prove a weighted version of this result. This is a special case of a general conjecture, on the weighted version of a duality phenomenon, between independence and domination in pairs of graphs.

\end{abstract}

\title{ Independence-Domination Duality in weighted graphs}
\author{Ron Aharoni}

\address{Department of Mathematics\\
 Technion, Haifa\\
 Israel 32000}

\email{Ron Aharoni: ra@tx.technion.ac.il}

\author{Irina Gorelik}

\address{Department of Mathematics\\
 Technion, Haifa\\
 Israel 32000}

\email{Irina Gorelik: iragor@tx.technion.ac.il}

\maketitle
\begin{section}{Introduction}

\subsection{Domination and collective domination}
All graphs in this paper are assumed to be simple, namely not containing parallel edges or loops. The (open) neighborhood of a vertex $v$ in a graph $G$, denoted  by $\tilde{N}(v)=\tilde{N}_G(v)$,   is the set of all vertices connected to $v$. Given a set $D$ of vertices we write  $\tilde{N}(D)$ for $\bigcup_{v \in D}\tilde{N}(v)$. Let  $N(D)=N_G(D) = \tilde{N}(D) \cup D$.
A set $D$ is said to be {\em dominating} if $N(D)=V$ and  {\em totally dominating} if
$\tilde{N}(D)=V$.
The minimal size of a dominating set is denoted by $\gamma(G)$, and the minimal size of a totally dominating set by
$\tilde{\gamma}(G)$.

There is a {\em collective} version of domination.
Given a system of graphs
$\zg=(G_1,\dots,G_k)$ on the same vertex set $V$, a system  $\zd=(D_1,\dots,D_k)$ of subsets of $V$ is said to be {\em collectively dominating} if $\bigcup_{i \le k}N_{G_i}(D_i)=V$. Let $\gamma_\cup(\zg)$ be the minimum of $\sum_{i \le k}|D_i|$ over all collectively dominating  systems.

\subsection{Independence and joint independence}
A set of vertices is said to be {\em independent} in $G$ if its elements are pairwise non-adjacent. The complex (closed down hypergraph) of independent sets in $G$ is denoted by $\zi(G)$. The {\em independence polytope} of $G$, denoted by
$IP(G)$, is the convex hull of the characteristic  vectors of the
 sets in $\zi(G)$. For a system of graphs
$\zg=(G_1,\dots,G_k)$ on  $V$  the
{\em  joint independence number}, $\alpha_\cap(\zg)$, is
$\max\{|I| : I \in \cap_{i \le k}\zi(G_i)\}$. The
{\em fractional joint independence number}, $\alpha_\cap^*(\zg)$, is
$\max\{\vec{x}\cdot\vec{1}:~~\vec{x}\in\bigcap_{i\le k} IP(G_i)\}$.\\

We shall mainly deal with the case  $k=2$. Let us first observe that it is possible to have   $\alpha_\cap^*(G_1,G_2)<\min(\alpha(G_1),\alpha(G_2))$.
\begin{example}
Let $G_1$ be obtained from the complete bipartite graph with respective sides $\{v_1, \ldots,v_6\}$ and $\{u_1,u_2\}$, by the addition of the edges $v_1v_2$, $v_3v_4$ and $u_1u_2$, and let $G_2=\bar{G_1}$. Then $\alpha(G_1)=\alpha(G_2)=4$, while  $\alpha_\cap^*(G_1,G_2)=2$, the optimal vector in $IP(G_1) \cap IP(G_2)$ being the constant $\frac{1}{4}$ vector.
\end{example}

 A graph $H$ is called a {\em partition graph} if it is the disjoint union of cliques. In a partition graph $\alpha=\gamma$.
The  union of two systems of disjoint cliques is the line graph of a bipartite graph, having the set of cliques in one system as one side of the graph, and the set of cliques in the other system as the other side, an edge connecting two vertices (namely, cliques in different systems) if they intersect.
Thus, by  K\"{o}nig's famous duality theorem \cite{konig}, we have:

\begin{theorem}\label{konig2}
If $G$ and $H$ are partition graphs on the same vertex set, then $$\alpha_\cap(G,H)=\gamma_\cup(G,H)$$.
\end{theorem}

There are graphs in which $\alpha >\gamma$, and thus
 equality does not necessarily  hold for general pairs $(G,H)$ of graphs,  even when $G=H$. On the other hand, since a maximal independent set is dominating, we have $\gamma(G) \le \alpha(G)$ in every graph $G$.   But  the corresponding inequality for pairs of graphs is not necessarily true, as the following example shows.
\begin{example}\label{noteq}
Let $G=P_4$, namely the path with $3$ edges on $4$ vertices, and let $H$ be its complement. Then  $\alpha_\cap(G,H)=1$
 and $\gamma_\cup(G,H)=2$, so $\alpha_\cap(G,H)<\gamma_\cup(G,H)$.
\end{example}

However, as was shown in  \cite{abhk}, if $\alpha_\cap$ is replaced by its
fractional version, then the non-trivial inequality in Theorem
\ref{konig2} does hold.

\begin{theorem}\label{inddom}
For any two graphs $G$ and $H$ on the same set of vertices we have $$\alpha_\cap^*(G,H)\geq\gamma_\cup(G,H)$$.
\end{theorem}

 In Example \ref{noteq}   $\vec{\frac{1}{2}}\in IP(G)\cap IP(H)$, and  $\alpha_\cap^*(G,H)=2$, so  $\alpha_\cap^*(G,H)=\gamma_\cup(G,H)$.

\begin{lemma}\label{fracit}
Let $\zv=(V_1, \ldots ,V_m)$ be a system of disjoint sets, let $\zi$ be the set of ranges of partial choice functions from $\zv$,  and let $V=\bigcup_{i \le m} V_i$.
Then
$$\{f: V \to \R^+ \mid \sum_{v \in V_j}f(v)\le 1 \text{~for~ every~} j \le m \}=conv(\{\chi_I \mid I \in \zi\})$$.
\end{lemma}

\begin{proof}
Obviously, the right hand side is contained in the left hand side. For the reverse containment,  let
$f: V \to \R^+$ be such that $\sum_{v \in V_j}f(v)\le 1$ for every $j \le m$, and assume for negation that it can be separated from  all functions $\chi_I$, $ I \in \zi$, namely there exists a vector $\vec{u}$ such that $\sum_{v\in V} u(v)f(v)\ge 1$, and $\sum_{v\in I}u(v)=\sum_{v\in V} u(v)\chi_I(v)<1$ for all
$I \in \zi$. Since $conv(\{\chi_I \mid I \in \zi\})$ is closed down, we may assume that $\vec{u}$ is non-negative. For each $j \le m$ let $v_j$ be such that $u(v_j)$ is maximal over all $v\in V_j$, and let $I=\{v_j \mid j \le m\}$. The fact that $\sum_{v\in I}u(v)< 1$ implies then that

$$\sum_{j \le m}\sum_{v\in V_j}u(v)f(v)\le \sum_{j \le m}u(v_j)\sum_{v \in V_j} f(v)\le \sum_{j \in V_j}u(v_j)< 1, $$

a contradiction.

\end{proof}

\subsection{Independent transversals}
When one graph in the pair $(G,H)$, say $H$, is a partition graph, the parameters $\alpha_\cap(G,H)$  and $\alpha^*_\cap(G,H)$ can be described   using   the terminology of so-called {\em independent transversals}.
Given a graph $G$ and a partition $\zv=(V_1, \ldots ,V_m)$ of $V(G)$, an  independent transversal (IT) is an independent set in $G$ consisting of the choice of one vertex from each set $V_i$. A {\em partial IT} is an independent set representing  some $V_i$'s (so, it is the independent range of a  partial choice function from $\zv$). A function $f : V \to \R^+$ is called a {\em partial fractional IT} if, when viewed as a vector, it belongs to  $IP(G)$, and $\sum_{v\in V_j}f(v)\le 1$ for all $j \le m$. If $\sum_{v\in V_j}f(v)=1$ for all $j \le m$ then $f$ is called a {\em fractional IT}. By  Lemma \ref{fracit} this means that $f \in IP(H) \cap IP(G)$, namely it is a jointly fractional independent set, where $\zv$ is the set of cliques in $H$.

For $I \subseteq [m]$ let $V_I=\bigcup_{i \in I}V_i$.

The following was proved in \cite{penny}:

\begin{theorem}\label{thm:penny}
If $\tg(G[V_I]) \ge 2|I|-1$ for every $I \subseteq [m]$ then there exists an IT.
\end{theorem}

Theorem \ref{inddom}, applied to the case in which $H$ is a partition graph, yields:

\begin{theorem}\label{fractr}
If $\gamma(G[V_I]) \ge |I|$ for every $I \subseteq [m]$ then there exists a fractional IT.
\end{theorem}

\section{Putting weights on the vertices}

In \cite{abz} a weighted version of Theorem \ref{thm:penny} was proved.
As is often the case with weighted versions, the motivation came from  decompositions: weighted results give, by duality, fractional decompositions results. It is conjectured that if $|V_i| \ge 2\Delta(G)$ then there exists a partition of $V(G)$ into $\max_{i \le m}|V_i|$ IT's. The weighted version of Theorem \ref{thm:penny} yielded the existence of a fractional such decomposition.

\begin{notation}
Given a real valued function $f$ on a set $S$,  and a set $A\subseteq
S$, define $f[A]=\sum_{a\in A}f(a)$. We also write $|f|=f[S]$ and we
call $|f|$  the {\em size} of $f$.
\end{notation}

\begin{definition}
Let $G=(V,E)$  be a graph, and let $w:V\to\mathbb{N}$  be a weight
function on $V$. We say that a function $f:V\to\mathbb{N} $ {\em
$w$-dominates} a set $U$ of vertices, if $f[N(u)]\geq w(u)$ for every $u\in U$. We say that $f$ is {\em $w$-dominating} if it $w$-dominates $V$.
 The {\em weighted domination number}
$\gamma^w(G)$  is   $\min\{|f| \mid f ~\text{is  $w$-dominating}\} $
\end{definition}

This definition extends to systems of graphs:

\begin{definition}
Let $\mathcal{G}=(G_1,\dots,G_k)$ be  a system of graphs on the same
vertex set $V$. Let $w:V\to\mathbb{N}$ be a non-negative weight function on $V$, and let $\zf=(f_i:~V\to\mathbb{N}, ~~i \le k)$ be a system of functions. We say that
$\zf$
$w$-dominates $\zg$ if   $\sum_{i=1}^k f_i[N_{G_i}(v)]\geq w(v)$ for every $v\in V$. The {\em
weighted collective domination number } is
$$\gamma_\cup^w(\mathcal{G})=\min\{\sum_{i=1}^k
|f_i|:(f_1,\dots,f_k)\; is\; w-dominating\}.$$

The extension of the independence parameter to the weighted case is also quite natural:

$$(\alpha_\cap^w)^*(\zg)=\max\{\sum_{v \in V}x(v)w(v) \mid \vec{x}\in\bigcap_{i=1}^k
IP(G_i)\}.$$
\end{definition}

The aim of this paper is to study the following possible   extension of Theorem \ref{inddom} to the weighted case.

\begin{Conjecture}\label{conj:main}
If $G$ and $H$ are graphs on the same vertex set $V$
then for any weight function $w:V\to\mathbb{N}$  we have
$$(\alpha_\cap^w)^*(G,H)\geq\gamma_\cup^w(G,H).$$
\end{Conjecture}

If $H=G$ then
the stronger $\alpha_\cap^w(G,G)  \geq\gamma_\cup^w(G,G)$ is true, namely:

\begin{lemma}
$\alpha^w(G) \ge \gamma^w(G)$.
\end{lemma}

\begin{proof}

We  have to exhibit a $w$-dominating function $f$ and an independent set $I$ with $|f|\leq w[I]$.

Let $V(G)=\{v_1,\dots,v_n\}$. We define a $w$-dominating function  $f:V\to \mathbb{N}$ inductively.
Let $f(v_1)=w(v_1)$. Having defined $f(v_1),\dots,f(v_{i-1})$ let $$f(v_i)= [w(v_i)-\sum_{v_j\in N(v_i),\; j<i}f(v_j)]^+$$
Clearly, $f$ is $w$-dominating.

We next find an independent set $I$ such that $w[I]\geq|f|$.
Let $v_{i_1}$ be the vertex that has  the maximal index over all the vertices in $V_1=V\cap supp(f)$. Since $f(v_{i_1})>0$, we have  $f[N(v_{i_1})]=w(v_{i_1})$.

Suppose that we have defined the sets of vertices $V_1,V_2,\dots, V_{k-1}$ and vertices $v_{i_1},\dots,v_{i_k}$ such that $v_{i_j}$ is the vertex whose index is maximal over all the vertices in $V_j$ where $V_j=V_{j-1}\setminus N(v_{i_{j-1}})$ for every $j=1,\dots,k-1$.

Let $V_k=V_{k-1}\setminus N(v_{i_{k-1}})$ and let $v_{i_k}$ be the vertex whose index is maximal over all the vertices in $V_k$. By the definition of $f$ we have $\sum_{v_j\in N(v_{i_k}),\; j<i_k}f(v_j)=w(v_{i_k})$, so $\sum_{v_j\in V_k\cap N(v_{i_k})}f(v_j)\leq w(v_{i_k})$.  We stop the process when $V_t=\emptyset$ for some $t$. In this case $I=\{v_{i_1},\dots,v_{i_{t-1}}\}$ is an independent set that satisfies $w[I]\geq |f|$ as desired.

\end{proof}

\section{The case of partition graphs}

The main result of this paper is:

\begin{theorem}\label{fractrconj}
Conjecture \ref{conj:main} is true if $H$ is a partition graph. Namely, if $H$ is a partition graph and $G$ is any graph, then
$$(\alpha_\cap^w)^*(G,H)\geq\gamma_\cup^w(G,H).$$
\end{theorem}

Let us first re-formulate the left hand  side of the inequality in terms of partitions.
For a partition $\zv=(V_1, \ldots, V_m)$ of the vertex set $V$ of a graph $G$, let \begin{equation} \label{ns} \nsw(G, \zv) =\max\{ \sum_{v \in V}w(v)f(v)\; \mid \; f ~\text{is~a~fractional~ partial ~IT}\}.\end{equation}

By Lemma \ref{fracit} we have:

\begin{lemma}\label{lem:param}
$(\alpha_\cap^w)^*(G,H)=\nsw(G, \zv)$.
\end{lemma}

Let us also  re-formulate  the right hand side using the terminology of partitions.
Given partition $\zv=(V_1, \ldots, V_m)$ of $V(G)$,  a pair of non-negative real valued functions $f$ on $V$ and $g$ on $[m]$ is said to be {\em collectively $w$-dominating}  if for every vertex $v \in V_i$ we have $g(i)+f[N(v)] \ge w(v)$. Let $\gamma^w(G, \zv)$  be the minimum of $|g|+|f|$  over all collectively $w$-dominating pairs of functions. In this terminology, $\gamma_\cup^w(G,H)=\gamma^w(G,\zv)$.  In addition, let $\tau^w(G, \zv)$  be the minimum of $|g|+\frac{|f|}{2}$  over all collectively $w$-dominating pairs of functions.

In \cite{abz} the following weighted version of Theorem \ref{fractr} was proved.

\begin{theorem}\label{thm:abz}
 $\nu^w(G,\zv) \ge \tau^w(G, \zv)$.
\end{theorem}

\begin{remark} Note the factor  $\frac{1}{2}$ difference between the definitions of  $\tau^w(G, \zv)$ and $\gamma^w(G,\zv)$. It mirrors the  factor $\frac{1}{2}$ difference (manifest in the factor $2$ in ``$2|I|-1$'') between the statements of
 Theorems \ref{thm:penny} and \ref{fractr}. The same factor appears in the weighted case: the difference between the integral and fractional versions is the $\frac{1}{2}$ factor hidden in the right hand sides
of  Theorems \ref{thm:abz} and of \ref{thm:partition} below. \end{remark}

By Lemma \ref{lem:param} the case of Conjecture \ref{conj:main} in which $H$ is a partition graph  is:

\begin{theorem}\label{thm:partition}
 ${(\nu^w)}^*(G,\zv) \ge \gamma^w(G, \zv)$, where $\zv$ is the partition of $V$ into cliques of $H$.
\end{theorem}


\begin{proof}
Note that  if $f = \sum_{I \in \zi(G)}x_I\chi_I$  then  $f[V_j]=\sum_{I\in\zi(G)}x_I|I\cap
V_j|$, and thus the constraints defining the linear program for $\nsw(G, \zv)$ are $\sum_{I\in\zi(G)}x_I|I\cap
V_j|\leq 1$ and $\sum_{I\in\zi(G)}x_I=1$.
\begin{assertion}
Let
 \begin{equation}\label{t}\nsw(G,\zv)=\max\{\sum_{I\in \zi(G)}x_I w[I]|\sum_{I\in\zi(G)}x_I|I\cap
V_j|\leq 1,\; \sum_{I\in\zi(G)}x_I\leq 1\}\end{equation}

\end{assertion}

\begin{proof}
Denote the right hand side by $t$.
If $f=\sum_{I\in\zi(G)}x_I\chi_I$ is an optimal solution of the linear program \eqref{ns} then $\nsw(G,\zv)=\sum_{v\in V}w(v)f(v)=\sum_{I\in \zi(G)}x_I w[I]$. Hence $\nsw(G,\zv)\leq t$.

On the other hand, suppose by negation that there exists an optimal solution of the linear program \eqref{t} that satisfies $\sum_{I\in\zi(G)}x_I=1-\epsilon$  for some $\epsilon>0$. Clearly $t>0$, and hence there exists an independent set $I_0$ such that $x_{I_0}>0$. Choose a vertex $v\in I_0$, and  define a  vector $\vec{x'}$  as follows. Let $x'_{I_0}=x_{I_0}-\epsilon$, $x'_{I_0\setminus\{v\}}=x'_{\{v\}}=\epsilon$ and  $x'_I=x_I$ otherwise. Note that the vector $x'$ satisfies   constrains of the linear program, but the weight of $\sum_{I\in\zi(G)}x'_I\chi_I$ is  $\sum_{I\in\zi(G)}x_Iw[I]+\epsilon$, contradicting the maximality of  the optimal solution. Hence this optimal solution is also a solution for the linear program \eqref{ns}, so, $t\leq \nsw(G,\zv)$ proving the desired equality.
\end{proof}
By LP  duality
$\nsw(G, \zv)$ is the minimum of $\sum_{j=0}^m y_j$ over all vectors $\vec{y}=(y_0,y_1,\dots,y_m)$ satisfying  $y_0+\sum_{j=1}^my_j |I\cap V_j| \geq w[I]$ for all  $I\in\zi(G)$.
Let $\vec{y}=(y_0,y_1,\dots,y_m)$ be a vector in which the minimum is attained, meaning that  $\sum_{j=0}^m y_j=\nsw(G,\zv)$,  and let  $g(j)=\lfloor y_j\rfloor$ for all $j \le m$. We define a new weight function $w_g$ by $w_g(v)=[w(v)-\lfloor y_{j(v)}\rfloor]^+$, where $j(v)$ is that $j$ for which $v \in V_j$. Let $V'=\{v \mid w_g(v)>0\}$ be the support of $w_g$, and let $G'=G[V']$.

For a number $s$ let $\{s\}$  be the fractional part of $s$, namely $\{s\}=s- \lfloor s \rfloor$.

\begin{assertion}
The vector $(y_0,\{y_1\},\dots,\{y_m\})$ is an optimal solution for the program dual to: $(\nu^{w_g})^*(G', \zv)$, namely
 $$y_0+\sum_{j=1}^m \{y_j\}=(\nu^{w_g})^*(G',\zv):=$$ $$\max\{\sum_{I\in\zi(G')}x_I w_g[I]\mid \;\sum_{I\in\zi(G')} x_I\leq1\; \text{and} \;  \forall j\;\sum_{I\in\zi(G')} x_I |I\cap V_j|\leq 1\}$$
\end{assertion}

\begin{proof}
Denote by $y$ the sum $y_0+\sum_{j=1}^m \{y_j\}$.
For every $v\in V'$, we have  $w_g(v)=w(v)-\lfloor y_{j(v)}\rfloor$, hence $y_0+\sum_{j=1}^m\{y_j\} |I\cap V_j|\geq w_g[I]$ for every $I\in\zi(G')\subseteq\zi(G)$, proving that $y\geq (\nu^{w_g})^*(G',\zv)$.

For the reverse inequality, assume for negation that  there exists a  solution $\vec{x}=(x_0,\dots,x_m)$ such that $\sum_{j=0}^m x_j<y$.
Then the vector $\vec{x}=(x_0,x_1+\lfloor y_1\rfloor,\dots,x_m+\lfloor y_m\rfloor)$ is a solution to the original problem that satisfies $x_0+\sum_{j=1}^m x_j+\lfloor y_j\rfloor<\sum_{j=0}^m y_j=\nsw(G,\zv)$, contradicting the optimality of $\vec{x}$.

\end{proof}

Since an optimal solution of the primary problem corresponding to the weight function $w_g$ satisfies $\sum_{I\in\mathcal{I}(G')}x_I=1$,  there exists a set $I$  such that $x_I>0$.

Let $I_0$ be a set of minimal weight in $supp(x)$.   Then \begin{equation}\label{I_0} w_g[I_0]=w_g[I_0]\sum_{I\in\zi(G')}x_I\leq \sum_{I\in\zi(G')}x_I w_g[I]=(\nu^{w_g})^*(G',\zv)\end{equation}

Let $h:V'\to\mathbb{N}$ defined by $h(v)=w_g(v)$  if $v\in I_0$ and $h(v)=0$ otherwise.
\begin{assertion}
The function $h$ is $w_g$-dominating in $G'$.
\end{assertion}

\begin{proof}
Suppose not. Then there exists a vertex $v\in V'$ such that $w_g(v)>h(v)+h[\tilde{N}(v)]=h(v)+w_g[\tilde{N}(v)\cap I_0]$. Clearly, $v\notin I_0$, hence $h(v)=0$. The set $I'=(I_0\setminus \tilde{N}(v))\cup \{v\}$ satisfies
$$w_g[I']=w_g[I_0\setminus \tilde{N}(v)]+w_g(v)>w_g[I_0\setminus \tilde{N}(v)]+w_g[I_0\cap \tilde{N}(v)]=w_g[I_0]$$ Since $w$ is an integral function,  $w_g$ is also integral, hence $w_g[I']=w_g[I_0]+k_v$ for some $k_v\geq 1$.

On the other hand, by the definition of the dual program $w_g[I']\leq y_0+\sum_{j=1}^m \{y_j\}|I'\cap V_j|$. In addition, since $x_{I_0}>0$, the complementary slackness conditions state that  equality holds in the corresponding constraint in the dual problem, i.e. $w_g[I_0]= y_0+\sum_{j=1}^m \{y_j\}|I_0\cap V_j|$. Hence,

$$k_v=w_g[I']-w_g[I_0]\leq \sum_{j=1}^m \{y_j\}(|I'\cap V_j|-|I_0\cap V_j|)=\{y_{j(v)}\}-\sum_{u\in N(v)\cap I_0}\{y_{j(u)}\}<1 $$

is a contradiction.
\end{proof}

Since $g$ dominates all  vertices $v\in V\setminus V'$, the pair $(g,h)$ is $w$-dominating, hence using \eqref{I_0}  we have $$\gamma^w(G,\zv)=\leq |g|+|h|=|g|+w_g[I_0]\leq |g|+(\nu^{w_g})^*(G',\zv)$$ $$=\sum_{j=1}^m \lfloor y_j\rfloor+y_0+\sum_{j=1}^m \{y_j\}=\sum_{j=0}^m y_j=\nsw(G, \zv) $$
as desired.

\end{proof}

\end{section}


\begin{thebibliography}{99}




\bibitem{abhk}
R. Aharoni, E. Berger, R. Holzman and O. Kfir, Independence - domination duality, {\it Jour. Combin. Th. Ser. B}, {\bf 98} (2008), 1259--1270.


\bibitem{abz} R. Aharoni, E. Berger and R. Ziv, Independent systems of representatives in weighted graphs, {\it Combinatorica}, {\bf 27}(3) (2007), 253--267.
\bibitem{ed}
J. Edmonds, Matroid intersection, {\it Ann. Discrete Math.}, {\bf 4} (1979), 39--49.

\bibitem{konig}
D. K\"onig, \"Uber Graphen und ihre Anwendung auf Determinanten-theorie und Mengenlehre {\it Math. Ann.}, {\bf 77} (1916), 453--465.

\bibitem{penny} P. E. Haxell, A condition for matchability in hypergraphs, {\em Graphs and
Combinatorics} {\bf 11} (1995), 245--248.
\end{thebibliography}
\end{document}